\documentclass[12pt]{amsart}

\newtheorem{theorem}{Theorem}[section]
\newtheorem{lemma}[theorem]{Lemma}
\newtheorem{proposition}[theorem]{Proposition}
\newtheorem{corollary}[theorem]{Corollary}   
\newtheorem{definition}[theorem]{Definition}
\newtheorem{example}[theorem]{Example}

\newtheorem{conjecture}[theorem]{Conjecture}

\numberwithin{equation}{section}

\usepackage{times}
\usepackage{enumerate}
\usepackage{mathrsfs}
\usepackage{tfrupee}
\usepackage{tikz}
\usetikzlibrary{chains,fit}
\usetikzlibrary{shapes,snakes}
\usetikzlibrary{graphs}
\usepackage{graphicx,adjustbox}
\usepackage{hyperref}
\usepackage{amsmath,amssymb}
\usepackage{amscd}
\usepackage{graphicx}
\usepackage[all]{xy}

\begin{document}

\title[Weighted Oriented Edge Ideals and Its Alexander Dual]{Cohen-Macaulay Weighted
Oriented Edge Ideals and its Alexander Dual}
\author{
Kamalesh Saha \and Indranath Sengupta
}
\date{}

%Author1
\address{\small \rm  Discipline of Mathematics, IIT Gandhinagar, Palaj, Gandhinagar, 
Gujarat 382355, INDIA.}
\email{kamalesh.saha@iitgn.ac.in}

%Author2
\address{\small \rm  Discipline of Mathematics, IIT Gandhinagar, Palaj, Gandhinagar, 
Gujarat 382355, INDIA.}
\email{indranathsg@iitgn.ac.in}
\thanks{The second author is the corresponding author; supported by the 
MATRICS research grant MTR/2018/000420, sponsored by the SERB, Government of India.}

\date{}



\allowdisplaybreaks

\begin{abstract}
The study of the edge ideal $I(D_{G})$ of a weighted oriented graph $D_{G}$ with underlying graph $G$ started in the 
context of Reed-Muller type codes. We generalize a Cohen-Macaulay construction for $I(D_{G})$, which Villarreal gave for edge ideals of simple graphs. We use this construction to classify all the Cohen-Macaulay weighted oriented edge ideals, whose underlying graph is a cycle. We show that the conjecture on Cohen-Macaulayness of $I(D_{G})$, proposed by Pitones et al. (2019), holds for $I(D_{C_{n}})$, where $C_{n}$ denotes the cycle of length $n$. Miller generalized the concept of Alexander dual ideals of square-free monomial ideals to arbitrary monomial ideals, and in that direction, we study the Alexander dual of $I(D_{G})$ and its conditions to be Cohen-Macaulay.
\end{abstract}
\medskip

\maketitle

\large{\section{Introduction}}
\medskip

Study of monomial ideals in terms of combinatorics is always an interesting 
topic in algebra. Many authors studied square-free monomial ideals, specially edge ideals of graphs extensively (see \cite{hhmon}, \cite{svv}, \cite{van}, \cite{vi}, \cite{vil}). Nowadays special types of ideals are being studied for the development of research in other areas, edge ideals of weighted oriented graph, defined below, is one of those.
\medskip

Let $G$ be simple graph. A \textit{weighted oriented graph} $D_{G}$, 
with the underlying graph $G$, is a directed graph with a weight function $w:V(D_{G})\longrightarrow \mathbb{N}\setminus \{0\}$ on the vertex 
set $V(D_{G})$, where $\mathbb{N}$ denotes the set of non negative integers. 
We write $w(j)$ instead of $w(x_{j})$. An edge of $D_{G}$ is denoted by 
the ordered pair $(x_{i},x_{j})$ to describe the direction of the edge 
from $x_{i}$ to $x_{j}$. 
\medskip

\begin{definition}{\rm 
Let $D_{G}$ be a weighted oriented graph on the vertex set $V(D_{G})=\{x_{1},\ldots,x_{n}\}$. Then the \textbf{edge ideal} of $D_{G}$ is denoted by $I(D_{G})$ and defined as
$$I(D_{G})=\big<\{x_{i}x_{j}^{w(j)}\mid (x_{i},x_{j})\in E(D_{G})\}\big>,$$
in the polynomial ring $A=K[x_{1},\ldots,x_{n}]$ over the field $K$.
}
\end{definition}
\medskip

The main purpose of studying weighted oriented edge ideal $I(D_{G})$ is the appearance of $I(D_{G})$ as the initial ideal of a vanishing ideal $I(\mathcal{X})$ in the study of Reed-Muller typed codes of coding theory (see \cite{cll}, \cite{mpv} for details). Study of $I(D_{G})$ helps to obtain some properties of Reed-Muller codes easily. For example, if $I(D_{G})$ is Cohen-Macaulay, then $I(\mathcal{X})$ is Cohen-Macaulay. To know about edge ideals of weighted oriented graphs in details, see \cite{gmsvv}, \cite{hlmrv}, \cite{prt}.
\medskip

In this paper, we show some constructions of Cohen-Macaulay 
edge ideals of weighted oriented graphs and also prove a sufficient 
condition for the Cohen-Macaulay property to hold for the Alexander 
dual of $I(D_{G})$. The paper is arranged in the following order: 
In section 2, we recall some definitions, notations and results 
from \cite{frob}, \cite{fg}, \cite{hlmrv}, \cite{prt}, \cite{peeva}. 
Villarreal has done some constructions of Cohen-Macaulay edge ideals 
of simple graphs (see \cite{vil} and \cite{vi}). We have 
generalised these results for weighted oriented edge ideals in section 3. In section 4, using the constructions of section 2, 
we classify all Cohen-Macaulay edge ideals of weighted oriented graphs whose underlying graphs are cycles. In \cite{prt}, 
the authors have proposed the following conjecture:
\medskip

\begin{conjecture}[\cite{prt}, Conjecture 53]\label{woconj}
Let $D_{G}$ be a weighted oriented graph. Then $I(D_{G})$ is Cohen-Macaulay if and only if $I(D_{G})$ is unmixed and $I(G)$ is Cohen-Macaulay.
\end{conjecture}
\medskip

\noindent In Theorem \ref{cyc3} and Theorem \ref{cyc5} (section 4), we prove that the above Conjecture \ref{woconj} holds good 
for all weighted oriented graphs $D_{G}$, where $G\simeq C_{n}$ is a cycle. 
The Alexander dual of an arbitrary monomial ideal was defined in \cite{mil}. 
It was proved by Lyubeznik in \cite{lg} that $I(G)^{\vee}$ is Cohen-Macaulay 
if and only $\overline{G}$ is chordal. We show in Theorem \ref{alx5}, in 
section 5, if $I(D_{G})^{\vee}$ is Cohen-Macaulay then 
$\overline{G}$ is chordal. However, the converse is not true. 
We prove Theorem \ref{alx4}, in section 5, to show that if 
$\overline{G}$ is chordal and $D_{G}$ satisfies a condition ($\ast$) (see section 5) 
then the Alexander dual $I(D_{G})^{\vee}$ is Cohen-Macaulay. 

\large{\section{Preliminaries}}
In this section, we introduce some terminologies and concepts from \cite{prt} and \cite{hlmrv}, which have been used 
extensively in our work. Let $D_{G}$ be a weighted oriented graph. Then a non-isolated vertex $v\in V(D_{G})$ is called a \textit{source} (respectively, a \textit{sink}) if $\{u,v\}\in E(G)$ implies $(v,u)\in E(D_{G})$ (respectively, $(u,v)\in E(D_{G})$). For a source vertex $v$, we always assume $w(v)=1$ as it does not change the ideal $I(D_{G})$. We denote by $\mathcal{N}(v)$, the set of vertices adjacent to a vertex $v$.
\medskip

\begin{definition}[\cite{prt} and \cite{hlmrv}]{\rm
Let $D_{G}$ be a weighted oriented graph. Corresponding to a vertex cover $C$ of $G$, define
\begin{align*}
& L_{1}(C)=\{x\in V(C) \mid \exists\, (x,y)\in E(D_{G})\,\, \text{such that}\,\, y\not\in C\}\\
& L_{2}(C)=\{x\in C\mid \mathcal{N}(x)\subseteq C\}\\
& L_{3}(C)=C\setminus (L_{1}(C)\cup L_{2}(C))
\end{align*}
A vertex cover $C$ of $G$ is called a strong vertex cover of $D_{G}$ if $C$ is either a minimal vertex cover of $G$ or for all $x\in L_{3}(C)$, there is $(y, x)\in E(D_{G})$ with $y\in L_{2}(C)\cup L_{3}(C)$ and $w(y)\neq 1$.}
\end{definition}

\begin{lemma}[\cite{prt}, Theorem 31]\label{unm}
$I(D_{G})$ is unmixed if and only if $I(G)$ is unmixed and $L_{3}(C)=\phi$, for any strong vertex cover $C$ of $D_{G}$.
\end{lemma}

Let $C$ be a vertex cover of $D_{G}$. The \textit{irreducible ideal associated to} $C$ is defined in \cite{prt} as 
$$I_{C}:=\big< L_{1}(C)\cup\{x_{j}^{w(j)}\mid x_{j}\in L_{2}(C)\cup L_{3}(C)\}\big>.$$

\begin{theorem}[\cite{prt}, Theorem 25]\label{prmdc}
Let $\mathcal{C}_{s}$ denote the set of strong vertex covers of $D_{G}$. Then the irredundant primary decomposition of 
$I(D_{G})$ is given by $I(D_{G}) = \bigcap_{C\in\mathcal{C}_{s}} I_{C}$. Moreover, $$\mathrm{Ass}\,(I(D_{G}))=\{P_{C}\mid P_{C}=\big< C\big>,\, C\in \mathcal{C}_{s}\}.$$
\end{theorem}

\noindent We say $D_{G}$ or $I(D_{G})$ is Cohen-Macaulay (respectively, unmixed) if the quotient ring $A/I(D_{G})$ is Cohen-Macaulay (respectively, if each associated prime of $I(D_{G})$ has 
the same height).

\begin{definition}{\rm
A vertex $v$ of a graph $G$ is called a 
\textit{simplicial} vertex if the subgraph of $G$
induced by the vertex set $\{v\}\cup \mathcal{N}(v)$ is a complete graph.
}
\end{definition}

A graph $G$ on $n$ vertices is said to have a 
\textit{perfect elimination ordering} if there is a linear ordering $(v_{1},\ldots,v_{n})$ of 
vertices of $G$ such that each $v_{i}$ is the 
simplicial vertex of the subgraph induced by 
the vertices $\{v_{1},\ldots,v_{i}\}$.

\begin{example}{\rm
Let us understand with an example.
\medskip

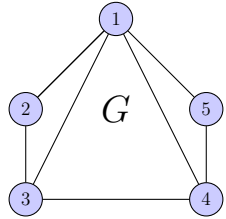
\begin{center}
\begin{tikzpicture}
  [scale=.4,auto=left,every node/.style={circle,scale=0.5}]
 
  \node[draw,fill=blue!20] (n1) at (0,0)  {$1$};
  \node[draw,fill=blue!20] (n2) at (-3,-3)  {$2$};
  \node[draw,fill=blue!20] (n3) at (-3,-6) {$3$};
   \node[draw,fill=blue!20] (n4) at (3,-6) {$4$};
   \node[draw,fill=blue!20] (n5) at (3,-3) {$5$};
 
\node[scale=2] (n6) at (0,-3){$G$};
 
  \foreach \from/\to in {n1/n2,n1/n3,n1/n4,n1/n2, n2/n3, n1/n5, n4/n5, n3/n4}
    \draw[] (\from) -- (\to);
   
\end{tikzpicture}
\end{center} 
\medskip

\noindent In the above graph $G$, the vertices $1,3,4$ are not simplicial but $2,5$ are simplicial. Moreover, $(1,3,4,2,5)$ is a perfect elimination ordering for $G$. }
\end{example}

\begin{definition}{\rm 
A graph $G$ is said to be \textit{chordal} if any induced cycle of $G$ with length $\geq 4$ has a chord. For example the above graph $G$ is chordal.
}
\end{definition}

\noindent The following theorem, proved by Fulkerson and
Gross in \cite{fg}, gives a necessary and sufficient condition 
for a graph to be chordal.

\begin{theorem}
A graph $G$ is chordal if and only if $G$ has a perfect elimination ordering.
\end{theorem}

In section 5, we use the technique of polarization to discuss 
the Cohen-Macaulay property of the Alexander dual of the 
weighted oriented edge ideals. Let us recall the notion 
of polarization from \cite{peeva} first. We use our own 
notation for convenience. Let 
$A=K[x_{1},\ldots,x_{n}]$; a monomial 
$x_{1}^{a_{1}}\cdots x_{n}^{a_{n}}$ in $A$ is 
written as $\mathbf{x^{a}}$, where $\mathbf{a}=(a_{1},\ldots,a_{n})\in \mathbb{N}^{n}$.

\begin{definition} [\cite{peeva}, Construction 21.7]{\rm 
The \textit{polarization} of the monomials of type $x_{i}^{a_{i}}$ is defined as $x_{i}^{a_{i}}(\mathrm{pol})=\prod_{j=1}^{a_{i}} x_{i,j}$ and the \textit{polarization} of $\mathbf{x^{a}}=x_{1}^{a_{1}}\cdots x_{n}^{a_{n}}$ is defined to be 
$$\mathbf{x^{a}}(\mathrm{pol})=x_{1}^{a_{1}}(\mathrm{pol})\cdots x_{n}^{a_{n}}(\mathrm{pol}).$$ 
For a monomial ideal $I=\big< \mathbf{x^{a_{1}}},\ldots,\mathbf{x^{a_{n}}}\big>\subseteq A$, the \textit{polarization} $I(\mathrm{pol})$ is defined to be the square-free monomial ideal
$$ I(\mathrm{pol})=\big< \mathbf{x^{a_{1}}}(\mathrm{pol}),\ldots,\mathbf{x^{a_{n}}}(\mathrm{pol})\big>$$
in the ring $A(\mathrm{
pol})=K[x_{i,j}\mid 1\leq i\leq n,\, 1\leq j\leq r_{i}]$ , where $r_{i}$ is the power of $x_{i}$ in \, $\mbox{lcm}(\mathbf{x^{a_{1}}},\ldots,\mathbf{x^{a_{n}}})$.
}
\end{definition}

\begin{theorem}[\cite{frob}]\label{cmpol}
Let $I$ be a monomial ideal of the polynomial ring $A$. Then $A/I$ is Cohen-Macaulay if and only if $A(\mathrm{pol})/I(\mathrm{pol})$ is Cohen-Macaulay.
\end{theorem}

\section{A Construction of Cohen-Macaulay Weighted Oriented Graphs}
\begin{theorem}[\cite{hlmrv}, Theorem 3.1]\label{wocm}
Let $D$ be a weighted oriented graph and $G$ be its underlying graph. Suppose that $G$ has a perfect matching $\lbrace x_{1},y_{1}\rbrace,\ldots,\lbrace x_{r},y_{r}\rbrace$, where $y_{i}$'s are leaf vertices. Then the following are equivalent:
\begin{enumerate}
\item[(a)] $D$ is a Cohen-Macaulay weighted oriented graph;
\item[(b)] $I(D)$ is unmixed; that is, all its associated primes have the same height;
\item[(c)] $w(x_{s})=1$ for every edge $(x_{s}, y_{s})$ of $D$.
\end{enumerate}
\end{theorem}

Start with an arbitrary weighted oriented graph $D_{G}$, whose underlying
graph is $G$ on the vertex set $\lbrace x_{1},\ldots,x_{n}\rbrace$. Let $G^{\prime}$ be the new graph obtained from 
$G$ by adding whiskers $\lbrace x_{1},y_{1}\rbrace$, 
$\ldots$, $\lbrace x_{n},y_{n}\rbrace$ to $G$. 
We now construct a 
new Cohen-Macaulay weighted oriented graph 
$D_{G^{\prime}}$, whose underlying graph is 
$G^{\prime}$. If $w(x_{i}) = 1$, then take any 
orientation $x_{i}$ to $y_{i}$ or $y_{i}$ to $x_{i}$ in $D_{G^{\prime}}$. If $w(x_{i})\not = 1$, then take the orientation $y_{i}$ to $x_{i}$ in $D_{G^{\prime}}$ and we may assign any weight to $y_{i}$. Then by theorem \ref{wocm} $D_{G^{\prime}}$ is a Cohen-Macaulay graph and $D_{G}$ is a weighted oriented subgraph of $D_{G^{\prime}}$.
\medskip

\noindent\textbf{First construction.} Let $D_{H}$ be a weighted oriented graph with the 
vertex set $V(D_{H})=\lbrace x_{1},\ldots,x_{n}, z, y\rbrace$ 
and the edge ideal 
$I(D_{H})$. Assume that $z$ is adjacent to $y$ with 
$\deg(z)\geq 2$ and $\deg(y) = 1$. We label the vertices 
of $D_{H}$ in such way that $x_{1},\ldots,x_{k},y$ 
are only adjacent to $z$, as shown in the
figure below. We also assume that $(x_{i},z)\in E(D_{H})$, for $i=1,\ldots,k$. 

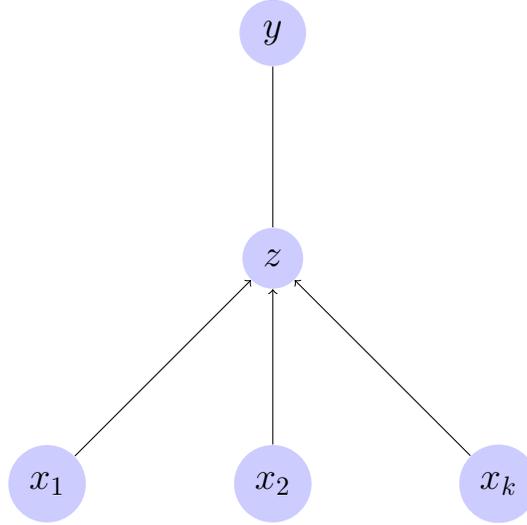
\begin{center}
\begin{tikzpicture}
  [scale=.6,auto=left,every node/.style={circle,fill=blue!20}]
  \node (n1) at (10,10) {$y$};
  \node (n2) at (10,5)  {$z$};
  \node (n4) at (5,0)  {$x_{1}$};
  \node (n3) at (10,0) {$x_{2}$};
   \node (n5) at (15,0) {$x_{k}$};

  \foreach \from/\to in {n2/n4,n2/n3,n2/n5}
    \draw[<-] (\from) -- (\to);
    \draw (n1) -- (n2);
    
\end{tikzpicture}
\end{center}

The next two results describe how the Cohen-Macaulay property of $D_{H}$ relates to that of the two weighted oriented subgraphs $D_{G} = D_{H}\setminus\lbrace z, y\rbrace$ and $D_{F} = D_{G}\setminus \lbrace x_{1},\ldots,x_{k}\rbrace$. We have $\big<I(D_{G}), x_{1},\ldots,x_{k}\big> = \big<I(D_{F}), x_{1},\ldots,x_{k}\big>=J\,\, \text{(say)}$. 
\medskip

Assume $H$ is unmixed with height of $I(D_{H})$ equal to $g+1$. Since $z$ is not isolated, there is a minimal prime $p$ over $I(D_{G})$ containing $\lbrace x_{1},\ldots,x_{k}\rbrace$ and such that $\mathrm{ht}\,(I(D_{G})) = \mathrm{ht}\,(p) = g$. Clearly, $k < n$ and deg$(x_{i})\geq 2$ for $i = 1,\ldots,k$.

\begin{proposition}\label{wocons1}
If $D_{H}$ is a Cohen-Macaulay graph, then $D_{F}$ and $D_{G}$ are Cohen-Macaulay graphs.
\end{proposition}

\begin{proof}
Set $A=K[x_{1},\ldots,x_{n}]$ and $R=A[y,z]$. Let us take $\mathrm{dim}\,(A/I(D_{G}))=d$. Then by (\cite{vil},  Proposition 3.1.23) there exists a homogeneous system of parameters $\lbrace f_{1},\ldots,f_{d}\rbrace$ for $A/I(D_{G})$, where $f_{i}\in A_{+}$ for all $i$. Now $\mathrm{dim}\,(A/(I(D_{G})+\big<f_{1},\ldots,f_{d}\big>))=0$ and $I(D_{G})+\big<f_{1},\ldots,f_{d}\big>\subset \big<x_{1},\ldots,x_{n}\big>$. So $\big<x_{1},\ldots,x_{n},y\big>\subset \big<x_{1},\ldots,x_{n},y,z\big>$ is a composition series
of prime ideal in R containing $I(D_{H})+(f_{1},\ldots,f_{d})$. Therefore 
$$\mathrm{dim}\,(R/(I(D_{H})+\big<f_{1},\ldots,f_{d}\big>))= 1 = d+1-d$$ and so we have $\lbrace f_{1},\ldots,f_{d}\rbrace$ as a part of system of parameters for $R/I(D_{H})$. Let $w(y)=r$ and $w(z)=t$. If $(y,z)\in E(D_{H})$,
then $yz^{t}\in I(D_{H})$. In this case we have
$$y(y - z^{t})+yz^{t}=y^{2}\,\,\mathrm{and}\,\, z^{t}(z^{t}-y)+yz^{t}= z^{2t}.$$
Using the above equalities we can say $(f_{1},\ldots,f_{d}, y - z^{t})$ is $M/I(D_{H})$-primary, where $M=(x_{1},\ldots,x_{n},y,z)$. Thus we can say $\lbrace f_{1},\ldots,f_{d},y-z^{t}\rbrace$ is a regular system of parameters for $R/I(D_{H})$. Hence $f_{1},\ldots,f_{d}$ is a regular sequence on $A/I(D_{G})$, that is, $D_{G}$ is Cohen-Macaulay. Similarly
if $(z,y)\in E(D_{H})$, then $zy^{r}\in I(D_{H})$. In this case $\lbrace f_{1},\ldots,f_{d},z-y^{r}\rbrace$ is a regular system of parameters for $R/I(D_{H})$ and similarly $D_{G}$ is Cohen-Macaulay. To show the Cohen-Macaulay property for $D_{F}$ we consider two cases.
\medskip

\noindent\textbf{Case I.} Let $(y,z)\in E(D_{H})$. In this case consider the exact sequence
{\fontsize{13}{14}\selectfont
$$ 0\longrightarrow R/(I(D_{H}):z^{t})[-t]\overset{z^{t}}{\longrightarrow}R/I(D_{H})\overset{\psi}{\longrightarrow} R/\big<I(D_{H}),z^{t}\big>\longrightarrow 0$$
}
which is equivalent to
{\fontsize{13}{14}\selectfont
$$ 0\longrightarrow R/\big<J,y\big>[-t]\overset{z^{t}}{\longrightarrow}R/I(D_{H})\overset{\psi}{\longrightarrow} R/\big<I(D_{G}),z^{t}\big>\longrightarrow 0$$
}
where the first map is multiplication by $z^{t}$ and $\psi$ is induced by a projection. Now $\mathrm{depth}\,(R/I(D_{H}))= \mathrm{dim}\,(R/I(D_{H}))=n-g+1$ and $\mathrm{depth}\,(R/\big<I(D_{G}),z^{t}\big>)=1+\mathrm{dim}\,(A/I(D_{G}))=n-g+1$, where $g=\mathrm{height}\,(I(D_{G}))$. Using the depth lemma we have 
$$n-g+1\leq \mathrm{depth}\,(R/\big<J,y\big>).$$
Also we have 
{\fontsize{13}{14}\selectfont
$$\mathrm{dim}\,(R/I(D_{H}))=\mathrm{max}\,\lbrace \mathrm{dim}\,(R/\big<J,y\big>),\mathrm{dim}\,(R/\big<I(D_{G}),z^{t}\big>)\rbrace,$$
}
which imply $\mathrm{dim}\,(R/\big<J,y\big>)\leq n-g+1$. Therefore $$R/\big<I(D_{G}),x_{1},\ldots,x_{k},y\big>=R/\big<I(D_{F}),x_{1},\ldots,x_{k},y\big>$$ is Cohen-Macaulay and so $D_{F}$ is Cohen-Macaulay.
\medskip

\noindent\textbf{Case II.} Let $(z,y)\in E(D_{H})$. In this case consider the exact sequence
{\fontsize{13}{14}\selectfont
$$ 0\longrightarrow R/(I(D_{H}):z^{t})[-t]\overset{z^{t}}{\longrightarrow}R/I(D_{H})\overset{\psi}{\longrightarrow} R/\big<I(D_{H}),z^{t}\big>\longrightarrow 0$$}
which is equivalent to
{\fontsize{13}{14}\selectfont
$$ 0\longrightarrow R/\big<J,y^{r}\big>[-t]\overset{z^{t}}{\longrightarrow}R/I(D_{H})\overset{\psi}{\longrightarrow} R/\big<I(D_{G}),z^{t},y^{r}z\big>\longrightarrow 0$$}
where the first map is multiplication by $z^{t}$ and $\psi$ is induced by a projection. In this case we have
$$\mathrm{depth}\,(R/\big<I(D_{G}),z^{t},y^{r}z\big>)=0+\mathrm{dim}\,(A/I(D_{G}))=n-g.$$ So using the depth lemma we get 
$$n-g+1\leq \mathrm{depth}\,(R/\big<J,y^{r}\big>).$$
Similarly as in case-I, we have $\mathrm{dim}\,(R/\big<J,y^{r}\big>)\leq n-g+1$. Therefore $R/\big<J,y^{r}\big>$ is Cohen-Macaulay and so $D_{F}$ is Cohen-Macaulay.

\end{proof}
\medskip

\begin{proposition}\label{wocons2}
If $D_{F}$ and $D_{G}$ are Cohen-Macaulay graphs, $(y,z)\in E(D)$ and $x_{1},\ldots,x_{k}$ are in some minimal vertex cover of $G$, then $D_{H}$ is a Cohen-Macaulay graph.
\end{proposition}
\begin{proof}
Consider the exact sequence of Proposition \ref{wocons1},
{\fontsize{13}{14}\selectfont
$$ 0\longrightarrow R/\big<J,y\big>[-t]\overset{z^{t}}{\longrightarrow}R/I(D_{H})\overset{\psi}{\longrightarrow} R/\big<I(D_{G}),z^{t}\big>\longrightarrow 0.$$}
Since $I(D_{G})$ is Cohen-Macaulay, $$\mathrm{depth}\,(R/\big<I(D_{G}),z^{t}\big>)= 1+\mathrm{dim}\,(A/I(D_{G}))=1+n-g,$$
 where $g=\mathrm{height}\,(I(D_{G}))$. As $x_{1},\ldots,x_{k}$ are in some minimal vertex cover of $G$, $\mathrm{height}\,(I(D_{H}))=g+1,\,\, \mathrm{height}\,(I(D_{F}))=g-k $ and so $\mathrm{dim}\,(R/I(D_{H}))=1+n-g$. Now 
\begin{align*}
 \mathrm{depth}\,(R/\big<J,y\big>)&=1+\mathrm{depth}\,(K[x_{k+1},\ldots,x_{n}]/I(D_{F}))\\&=1+\mathrm{dim}\,(K[x_{k+1},\ldots,x_{n}]/I(D_{F}))\\&=1+(n-k)-(g-k)\\&=1+n-g.
 \end{align*}
Therefore by depth lemma we have $$\mathrm{depth}\,(R/I(D_{H}))\geq 1+n-g,$$ which imply $D_{H}$ is Cohen-Macaulay graph.
\end{proof}
\begin{corollary}\label{wocons3}
If $D_{G}$ is Cohen-Macaulay, $(y,z)\in E(D_{H})$ and $\lbrace x_{1},\ldots, x_{k}\rbrace$ is a minimal vertex cover of $G$, then $D_{H}$ is Cohen-Macaulay.
\end{corollary}
\begin{proof}
Note that in this case $D_{F}$ is Cohen-Macaulay because $I(D_{F})=(0)$. Hence the proof follows from Proposition \ref{wocons2}.
\end{proof}
\medskip

\textbf{Special Case of First Construction}: We are taking same configuration as before with $w(z)=1$ but direction of edges between $x_{i}$ and $z$ may be anything for $i=\{1,\ldots k\}$. Without loss of generality we assume $(x_{1},z),\ldots,(x_{s},z)\in E(D_{H})$ and $(z,x_{s+1}),\ldots,(z,x_{k})\in E(D_{H})$.
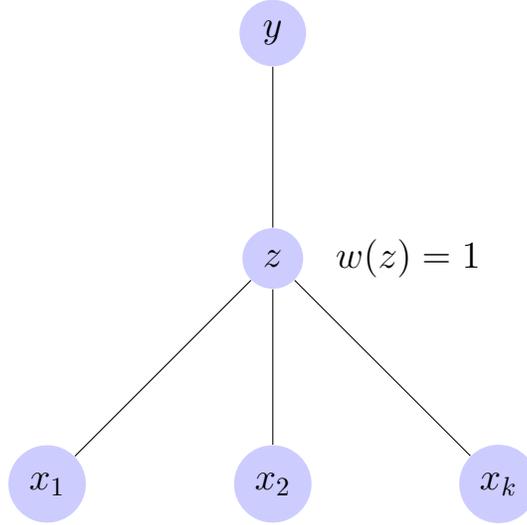
\medskip

\begin{center}
\begin{tikzpicture}
  [scale=.6,auto=left,every node/.style={circle,fill=blue!20}]
  \node (n1) at (10,10) {$y$};
  \node (n2) at (10,5)  {$z$};
  \node (n4) at (5,0)  {$x_{1}$};
  \node (n3) at (10,0) {$x_{2}$};
   \node (n5) at (15,0) {$x_{k}$};
\node[fill=white] (n6) at (13,5) {$w(z)=1$};
  \foreach \from/\to in {n2/n4,n2/n3,n2/n5}
    \draw (\from) -- (\to);
    \draw (n1) -- (n2);
    
\end{tikzpicture}
\end{center}
\medskip

\begin{proposition}\label{wocons4}
If $D_{H}$ is a Cohen-Macaulay graph, then $D_{G}$ is Cohen-Macaulay graphs. Moreover, if for any vertex $v$ of $D_{F}$ adjacent to $x_{j}$ where $s+1\leq j\leq k$ we have $(v,x_{j})\in E(D_{G})$, then $D_{F}$ is also Cohen-Macaulay graphs.
\end{proposition}

\begin{proof}
By same argument as in proof of Proposition \ref{wocons1}, we get $\lbrace f_{1},\ldots,f_{d}\rbrace$ as a part of system of parameters for $R/I(D_{H})$, where $f_{i}\in A_{+}$ for all $i$. Again following the proof of Proposition \ref{wocons1}, it is easy to see that $\lbrace f_{1},\ldots,f_{d},y-z\rbrace$ is a regular system of parameters for $R/I(D_{H})$ when $(y,z)\in E(D_{H})$ and $\lbrace f_{1},\ldots,f_{d},z-y^{r}\rbrace$ as a regular system of parameters for $R/I(D_{H})$ when $(z,y)\in E(D_{H})$. In either cases $\lbrace f_{1},\ldots,f_{d}\}$ is a regular sequence on $A/I(D_{G})$ and hence $D_{G}$ is Cohen-Macaulay.
\medskip

To show the Cohen-Macaulay property for $D_{F}$ we consider the exact sequence
{\fontsize{13}{14}\selectfont
$$ 0\longrightarrow R/(I(D_{H}):z)[-1]\overset{z}{\longrightarrow}R/I(D_{H})\overset{\psi}{\longrightarrow} R/\big<I(D_{H}),z\big>\longrightarrow 0$$}
which is equivalent to
$$ 0\longrightarrow R/J[-1]\overset{z}{\longrightarrow}R/I(D_{H})\overset{\psi}{\longrightarrow} R/\big<I(D_{G}),z\big>\longrightarrow 0$$
where
{\fontsize{13}{14}\selectfont
$$J= 
\begin{cases}
    \big<I(D_{G}),x_{1},\ldots,x_{s},x_{s+1}^{w(x_{s+1})},\ldots,x_{k}^{w(x_{k})},y\big>,& \mathrm{if } (y,z)\in E(D_{H})\\
    \big<I(D_{G}),x_{1},\ldots,x_{s},x_{s+1}^{w(x_{s+1})},\ldots,x_{k}^{w(x_{k})},y^{r}\big>,              & \mathrm{if} (z,y)\in E(D_{H})
\end{cases}.$$
}

  Since $D_{G}$ and $D_{H}$ are Cohen-Macaulay, $$\mathrm{depth}\,(R/I(D_{H}))=n-g+1=\mathrm{depth}\,(R/\big<I(D_{G}),z\big>),$$ where $g=\mathrm{height}\,(I(D_{G})).$ Using the depth lemma we have 
$$n-g+1\leq \mathrm{depth}\,(R/J).$$
Also we have 
$$\mathrm{dim}\,(R/I(D_{H}))=\mathrm{max}\,\lbrace \mathrm{dim}\,(R/J),\mathrm{dim}\,(R/\big<I(D_{G}),z\big>)\rbrace,$$
which imply $\mathrm{dim}\,(R/J)\leq n-g+1$. Therefore $R/J$ is Cohen-Macaulay. Now by the given condition on the vertices of $D_{F}$ we have 
{\fontsize{13}{14}\selectfont
$$J= 
\begin{cases}
    \big<I(D_{F}),x_{1},\ldots,x_{s},x_{s+1}^{w(x_{s+1})},\ldots,x_{k}^{w(x_{k})},y\big>,& \mathrm{if } (y,z)\in E(D_{H})\\
    \big<I(D_{F}),x_{1},\ldots,x_{s},x_{s+1}^{w(x_{s+1})},\ldots,x_{k}^{w(x_{k})},y^{r}\big>,              & \mathrm{if} (z,y)\in E(D_{H})
\end{cases}$$}
 and so $D_{F}$ is Cohen-Macaulay.
\end{proof}
\medskip

\begin{proposition}\label{wocons5}
If $D_{F}$ and $D_{G}$ are Cohen-Macaulay graphs, $x_{1},\ldots,x_{k}$ are in some minimal vertex cover of $G$ and for any vertex $v$ of $D_{F}$ adjacent to $x_{j}$ where $s+1\leq j\leq k$ we have $(v,x_{j})\in E(D_{G})$, then $D_{H}$ is a Cohen-Macaulay graph.
\end{proposition}

\begin{proof}
Similar as in Proposition \ref{wocons4}, we have the exact sequence
$$ 0\longrightarrow R/J[-1]\overset{z}{\longrightarrow}R/I(D_{H})\overset{\psi}{\longrightarrow} R/\big<I(D_{G}),z\big>\longrightarrow 0$$
where
{\fontsize{13}{14}\selectfont
$$J= 
\begin{cases}
    \big<I(D_{G}),x_{1},\ldots,x_{s},x_{s+1}^{w(x_{s+1})},\ldots,x_{k}^{w(x_{k})},y\big>,& \mathrm{if } (y,z)\in E(D_{H})\\
    \big<I(D_{G}),x_{1},\ldots,x_{s},x_{s+1}^{w(x_{s+1})},\ldots,x_{k}^{w(x_{k})},y^{r}\big>,              & \mathrm{if} (z,y)\in E(D_{H})
\end{cases}.$$
}
By the given condition on the vertices of $D_{F}$ we have 
{\fontsize{13}{14}\selectfont
$$J= 
\begin{cases}
    \big<I(D_{F}),x_{1},\ldots,x_{s},x_{s+1}^{w(x_{s+1})},\ldots,x_{k}^{w(x_{k})},y\big>,& \mathrm{if } (y,z)\in E(D_{H})\\
    \big<I(D_{F}),x_{1},\ldots,x_{s},x_{s+1}^{w(x_{s+1})},\ldots,x_{k}^{w(x_{k})},y^{r}\big>,              & \mathrm{if} (z,y)\in E(D_{H})
\end{cases}.$$
}
Since $I(D_{G})$ is Cohen-Macaulay, $$\mathrm{depth}\,(R/\big<I(D_{G}),z\big>)= 1+n-g,$$
 where $g=\mathrm{height}\,(I(D_{G}))$. As $x_{1},\ldots,x_{k}$ are in some minimal vertex cover of $G$, $\mathrm{height}\,(I(D_{H})=g+1,\,\, \mathrm{height}\,(I(D_{F}))=g-k $ and so $\mathrm{dim}\,(R/I(D_{H}))=1+n-g$. Now
\begin{align*}
 \mathrm{depth}\,(R/J)&=1+\mathrm{depth}\,(K[x_{k+1},\ldots,x_{n}]/I(D_{F}))\\&=1+\mathrm{dim}\,(K[x_{k+1},\ldots,x_{n}]/I(D_{F}))\\&=1+(n-k)-(g-k)\\&=1+n-g.
 \end{align*}
Hence by depth lemma we have $\mathrm{depth}\,(R/I(D_{H}))\geq 1+n-g$, which imply $D_{H}$ is Cohen-Macaulay graph.
\end{proof}
\medskip

\begin{corollary}\label{wocons6}
If $D_{G}$ is Cohen-Macaulay, $\lbrace x_{1},\ldots, x_{k}\rbrace$ is a minimal vertex cover of $G$ and for any vertex $v$ of $D_{F}$ adjacent to $x_{j}$ where $s+1\leq j\leq k$ we have $(v,x_{j})\in E(D_{G})$, then $D_{H}$ is Cohen-Macaulay.
\end{corollary}
\begin{proof}
Note that in this case $D_{F}$ is Cohen-Macaulay because $I(D_{F})=(0)$. Hence the proof follows from Proposition \ref{wocons5}.
\end{proof}
\medskip

\textbf{Second construction:} For the second construction we change our notation. Let $D_{H}$ be a weighted oriented graph on the vertex set $V(D_{H})=\lbrace x_{1},\ldots, x_{n},z\rbrace$. Let $\lbrace x_{1},\ldots, x_{k}\rbrace$ be the vertices of $D_{H}$ adjacent to $z$ and $(x_{i},z)\in E(D_{H})$ for $i=1,\ldots,k$, as shown in the figure below. Remembering the first construction we may assume $\mathrm{deg}\,(x_{i})\geq 2$ for $i=1,\ldots, k$ and $\mathrm{deg}\,(z)\geq 2$. We set $D_{G}=D_{H}\setminus\lbrace z\rbrace$ and $D_{F}=D_{G}\setminus\lbrace x_{1},\ldots, x_{k}\rbrace$. Note that
$$\big<I(D_{G}), x_{1},\ldots, x_{k}\big>=\big<I(D_{F}),x_{1},\ldots, x_{k}\big>=J\,\,(\text{say}).$$

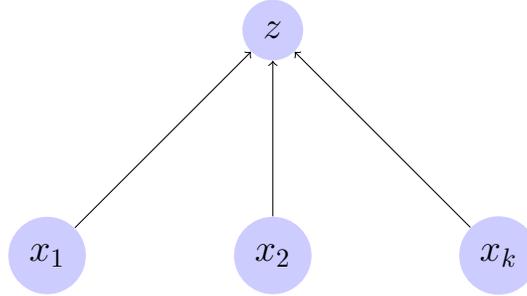
\begin{center}
\begin{tikzpicture}
  [scale=.6,auto=left,every node/.style={circle,fill=blue!20}]
  \node (n2) at (10,5)  {$z$};
  \node (n4) at (5,0)  {$x_{1}$};
  \node (n3) at (10,0) {$x_{2}$};
   \node (n5) at (15,0) {$x_{k}$};

  \foreach \from/\to in {n2/n4,n2/n3,n2/n5}
    \draw[<-] (\from) -- (\to);
\end{tikzpicture}
\end{center}
\medskip

\begin{proposition}\label{wocons7}
If $D_{H}$ is a Cohen-Macaulay graph, then $D_{F}$ is a Cohen-Macaulay graph.
\end{proposition}
\begin{proof}
We set $R=K[x_{1},\ldots, x_{n}, z]$, $A=K[x_{1},\ldots, x_{n}]$ and $\mathrm{ht}\,(I(D_{H}))=g +1$. Since $D_{H}$ is unmixed, from Lemma \ref{unm} we have $L_{3}(C)=\phi$ for any strong vertex cover $C$ of $D_{H}$. So $\lbrace z,x_{1},\ldots,x_{k}\rbrace$ can not be in any strong vertex cover of $D_{H}$. Hence the polynomial $f=z-x_{1}-\cdots - x_{k}$ is regular on $R/I(D_{H})$ as it is clearly not contained in any associated prime of $I(D_{H})$. Therefore there is a sequence $\lbrace f,f_{1},\ldots, f_{m}\rbrace$ regular on $R/I(D_{H})$ so that $\lbrace f_{1},\ldots, f_{m}\rbrace\subset A_{+}$, where $m=n-g-1$. Observe that $\lbrace f_{1},\ldots, f_{m}\rbrace$ is in fact a regular sequence on
$A/I(D_{G})$, which gives $$\mathrm{depth}\,(A/I(D_{G}))=\mathrm{depth}\,(R/\big<I(D_{G}),z^{d}\big>)\geq n-g-1,$$ where $d=w(z)$. Next, we use the exact sequence
{\fontsize{13}{14}\selectfont
$$ 0\longrightarrow R/(I(D_{H}):z^{d})[-d]\overset{z^{d}}{\longrightarrow}R/I(D_{H})\overset{\psi}{\longrightarrow} R/\big<I(D_{H}),z^{d}\big>\longrightarrow 0$$}
which is equivalent to
{\fontsize{13}{14}\selectfont
$$ 0\longrightarrow R/J[-d]\overset{z^{d}}{\longrightarrow}R/I(D_{H})\overset{\psi}{\longrightarrow} R/\big<I(D_{G}),z^{d}\big>\longrightarrow 0.$$
}
Now $\mathrm{depth}\,(R/I(D_{H}))=\mathrm{dim}\,(R/I(D_{H}))=n-g$. By depth lemma we get $\mathrm{depth}\,(R/J)\geq n-g$. Since there is a minimal vertex cover of $D_{H}$ containing $\lbrace x_{1},\ldots,x_{k}\rbrace$, we have $\mathrm{height}\,(J)=g+1$. Therefore $R/J$ is Cohen-Macaulay and hence $D_{F}$ is Cohen-Macaulay.
\end{proof}
\medskip

\begin{proposition}\label{wocons8}
Assume that $\lbrace x_{1},\ldots, x_{k}\rbrace$ is not contained in any minimal vertex cover of $D_{G}$ and $$\mathrm{height}\,(\big<I(D_{G}), x_{1},\ldots,x_{k}\big>) =\mathrm{height}\,(I(D_{G}))+1. $$ If $D_{F}$ and $D_{G}$ are Cohen-Macaulay, then $D_{H}$ is Cohen-Macaulay.
\end{proposition}
\begin{proof}
The assumption on $\lbrace x_{1},\ldots, x_{k}\rbrace$ forces $$\mathrm{height}\,(I(D_{H}))=\mathrm{height}\,(I(D_{G}))+1.$$ Consider the exact sequence
$$ 0\longrightarrow R/J[-d]\overset{z^{d}}{\longrightarrow}R/I(D_{H})\overset{\psi}{\longrightarrow} R/\big<I(D_{G}),z^{d}\big>\longrightarrow 0.$$
Now $\mathrm{depth}\,(R/\big < I(D_{G}),z^{d}\big >)=\mathrm{dim}\,(R/\big<I(D_{G}),z^{d}\big>)=n-g$, where $g=\mathrm{height}\,(I(D_{G}))$. Since $D_{F}$ is Cohen-Macaulay, $R/\big<I(D_{F}),x_{1},\ldots,x_{k}\big>=R/J$ is Cohen-Macaulay. So we have $$\mathrm{depth}\,(R/J)=\mathrm{dim}\,(R/J)=n-g.$$ Using depth lemma we get from the above exact sequence $$\mathrm{depth}\,(R/I(D_{H}))=n-g=\mathrm{dim}\,(R/I(D_{H})).$$ Hence $D_{H}$ is Cohen-Macaulay.
\end{proof}
\medskip

\begin{corollary}\label{wocons9}
If $D_{G}$ is Cohen-Macaulay and $\lbrace x_{1},\ldots, x_{k-1}\rbrace$ is a minimal vertex cover for $D_{G}$, then $D_{H}$ is Cohen-Macaulay.
\end{corollary}
\begin{proof}
Note that $I(D_{F})=0$ in this case and hence the proof follows.
\end{proof}
\medskip

\textbf{Special Case of Second Construction}:
We are taking same configuration as in second construction with $w(z)=1$ but direction of edges between $x_{i}$ and $z$ may be anything for $i=\{1,\ldots k\}$. We may assume $(x_{1},z),\ldots,(x_{s},z)\in E(D_{H})$ and $(z,x_{s+1}),\ldots,(z,x_{k})\in E(D_{H})$.
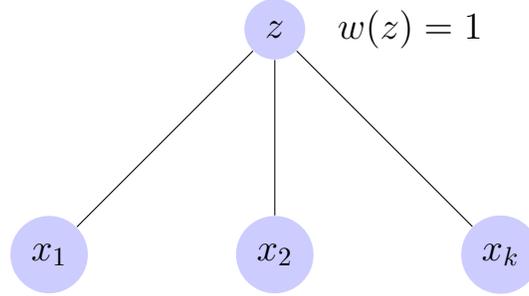
\medskip

\begin{center}
\begin{tikzpicture}
  [scale=.6,auto=left,every node/.style={circle,fill=blue!20}]
  \node (n2) at (10,5)  {$z$};
  \node (n4) at (5,0)  {$x_{1}$};
  \node (n3) at (10,0) {$x_{2}$};
   \node (n5) at (15,0) {$x_{k}$};
\node[fill=white] (n6) at (13,5) {$w(z)=1$};

  \foreach \from/\to in {n2/n4,n2/n3,n2/n5}
    \draw[] (\from) -- (\to);
\end{tikzpicture}
\end{center}
\medskip

\begin{proposition}\label{wocons10}
If $D_{H}$ is a Cohen-Macaulay graph and for any vertex $v$ of $D_{F}$ adjacent to $x_{j}$ where $s+1\leq j\leq k$ we have $(v,x_{j})\in E(D_{G})$, then $D_{F}$ is a Cohen-Macaulay graph.
\end{proposition}

\begin{proof}
Following the same arguments and notations as in proof of Propostion \ref{wocons7} we get $f=z-x_{1}-\cdots - x_{k}$ is regular on $R/I(D_{H})$. Therefore there is a sequence $\lbrace f,f_{1},\ldots, f_{m}\rbrace$ regular on $R/I(D_{H})$ so that $\lbrace f_{1},\ldots, f_{m}\rbrace\subset A_{+}$, where $m=n-g-1$. Observe that $\lbrace f_{1},\ldots, f_{m}\rbrace$ is in fact a regular sequence on
$A/I(D_{G})$, which gives $$\mathrm{depth}\,(A/I(D_{G}))=\mathrm{depth}\,(R/\big<I(D_{G}),z\big>)\geq n-g-1.$$ Next, we use the exact sequence
{\fontsize{13}{14}\selectfont
$$ 0\longrightarrow R/(I(D_{H}):z)[-1]\overset{z}{\longrightarrow}R/I(D_{H})\overset{\psi}{\longrightarrow} R/\big<I(D_{H}),z\big>\longrightarrow 0$$
}
which is equivalent to
$$ 0\longrightarrow R/J[-1]\overset{z}{\longrightarrow}R/I(D_{H})\overset{\psi}{\longrightarrow} R/\big<I(D_{G}),z\big>\longrightarrow 0,$$
where $J=\big<I(D_{G}),x_{1},\ldots,x_{s},x_{s+1}^{w(x_{s+1})},\ldots,x_{k}^{w(x_{k})}\big>$.
\medskip

Now $\mathrm{depth}\,(R/I(D_{H}))=\mathrm{dim}\,(R/I(D_{H}))=n-g$. By depth lemma we have $\mathrm{depth}\,(R/J)\geq n-g$. Again by the given condition on the vertices of $D_{F}$ we have $$J=\big<I(D_{F}),x_{1},\ldots,x_{s},x_{s+1}^{w(x_{s+1})},\ldots,x_{k}^{w(x_{k})}\big>.$$ There is a minimal vertex cover of $D_{H}$ containing $\lbrace x_{1},\ldots,x_{k}\rbrace$ and so $\mathrm{height}\,(J)=g+1$. Therefore $R/J$ is Cohen-Macaulay and thus $D_{F}$ is Cohen-Macaulay.
\end{proof}
\medskip

\begin{proposition}\label{wocons11}
Assume that $\lbrace x_{1},\ldots, x_{k}\rbrace$ are not contained in any minimal vertex cover of $D_{G}$ and $$\mathrm{height}\,(J)=\mathrm{height}\,(I(D_{G}))+1,$$ where $J=\big<I(D_{G}),x_{1},\ldots,x_{s},x_{s+1}^{w(x_{s+1})},\ldots,x_{k}^{w(x_{k})}\big>$. Also assume  for any vertex $v$ of $D_{F}$ adjacent to $x_{j}$ where $s+1\leq j\leq k$ we have $(v,x_{j})\in E(D_{G})$. If $D_{F}$ and $D_{G}$ are Cohen-Macaulay, then $D_{H}$ is Cohen-Macaulay.
\end{proposition}
\begin{proof}
The assumption on $\lbrace x_{1},\ldots, x_{k}\rbrace$ forces $$\mathrm{ht}\,(I(D_{H}))=\mathrm{ht}\,(I(D_{G}))+1.$$ Consider the exact sequence
$$ 0\longrightarrow R/J[-1]\overset{z}{\longrightarrow}R/I(D_{H})\overset{\psi}{\longrightarrow} R/\big<I(D_{G}),z\big>\longrightarrow 0.$$ 
By the given condition on the vertices of $D_{F}$ we also have $$J=\big<I(D_{F}),x_{1},\ldots,x_{s},x_{s+1}^{w(x_{s+1})},\ldots,x_{k}^{w(x_{k})}\big>.$$
Since $D_{G}$ is cohen-macaulay, $$\mathrm{depth}\,(R/(I(D_{G}),z)=\mathrm{dim}\,(R/(I(D_{G}),z)=n-g,$$ where $g=\mathrm{height}\,(I(D_{G}))$. Again $D_{F}$ is Cohen-Macaulay implies $R/J$ is Cohen-Macaulay. So we have $\mathrm{depth}\,(R/J)=n-g$. Using depth lemma we get from the above exact sequence $$\mathrm{depth}\,(R/I(D_{H}))=n-g=\mathrm{dim}\,(R/I(D_{H})).$$ Hence $D_{H}$ is Cohen-Macaulay.
\end{proof}
\medskip

\begin{corollary}\label{wocons12}
If $D_{G}$ is Cohen-Macaulay, $\lbrace x_{1},\ldots, x_{k-1}\rbrace$ is a minimal vertex cover for $D_{G}$ and for any vertex $v$ of $D_{F}$ adjacent to $x_{j}$ where $s+1\leq j\leq k$ we have $(v,x_{j})\in E(D_{G})$, then $D_{H}$ is Cohen-Macaulay.
\end{corollary}
\begin{proof}
Note that $I(D_{F})=0$ in this case and hence the proof follows by Proposition \ref{wocons11}.
\end{proof}
\medskip

\section{Cohen-Macaulay Edge Ideals of Weighted Oriented Cycles}
\medskip

\begin{proposition}{\rm (\cite{prt}, Proposition 54)}\label{cyc1}
If $G$ is a path, then the following conditions are equivalent:
\begin{enumerate}[(1)]
\item $R/I(D_{G})$ is Cohen-Macaulay.
\item $I(D_{G})$ is unmixed.
\item $\vert V(G)\vert = 2$ or $\vert V(G)\vert = 4$. In the second case, if $(x_{2}, x_{1})\in E(D_{G})$ or $(x_{3}, x_{4})\in E(D_{G})$,
then $w(x_{2}) = 1$ or $w(x_{3}) = 1$ respectively.
\end{enumerate}
\end{proposition}
\medskip

\begin{proposition}\label{cyc2}
Let $D_{G}$ be a Cohen-Macaulay weighted oriented graph whose underlying graph $G$ is a cycle. Then $G$ is either $C_{3}$ or $C_{5}$.
\end{proposition}

\begin{proof}
$I(D_{G})$ is Cohen-Macaulay implies $I(G)$ is Cohen-Macaulay by (\cite{prt}, Proposition 51). Since $G$ is a cycle, by (\cite{vil}, Corollary 7.3.19) we have $I(G)$ is Cohen-Macaulay if and only if $G$ is $C_{3}$ or $C_{5}$. Hence the result follows.
\end{proof}
\medskip

\begin{theorem}\label{cyc3}
$I(D_{C_{3}})$ is Cohen-Macaulay if and only if there exists $x\in V(D_{C_{3}})$ such that $w(x)=1$.
\end{theorem}

\begin{proof}
$"\Longrightarrow"$ $I(D_{C_{3}})$ is Cohen-Macaulay implies $I(D_{C_{3}})$ is unmixed. Then by (\cite{prt}, Theorem 49) there exists $x\in V(D_{C_{3}})$ such that $w(x)=1$.
\medskip

$"\Longleftarrow"$ Compare $D_{C_{3}}$ with the special case of second construction and consider $z=x$. Now $D_{C_{3}}\setminus \{x\}$ is a weighted oriented path of two vertices and so it is Cohen-Macaulay by Proposition \ref{cyc1}. Hence from Corollary \ref{wocons12} we have $I(D_{C_{3}})$ is Cohen-Macaulay.
\end{proof}
\medskip

\begin{center}
\begin{tikzpicture}
  [scale=.4,auto=left,every node/.style={circle,scale=0.5}]
 
  \node[draw,fill=blue!20] (n1) at (-2.5,8)  {$x_{1}$};
  \node[draw,fill=blue!20] (n2) at (2.5,8)  {$x_{2}$};
  \node[draw,fill=blue!20] (n3) at (2.5,3.5) {$x_{3}$};
   \node[draw,fill=blue!20] (n4) at (0,0) {$x_{4}$};
   \node[draw,fill=blue!20] (n5) at (-2.5,3.5) {$x_{5}$};
 \node (n6) at (-4,2.5){$w(x_{5})\neq 1$};
 \node (n7) at (4,2.5){$w(x_{3})\neq 1$};
 \node (n8) at (-2.5,9){$1$};
 \node (n9) at (2.5,9){$1$};
 \node (n10) at (0,-1){$1$};
\node[scale=2] (n11) at (0,5.5){$D_{1}$};
  \foreach \from/\to in {n4/n3,n4/n5,n3/n2,n5/n1}
    \draw[->] (\from) -- (\to);
    \draw (n1) -- (n2);
    
     \node[draw,fill=blue!20] (m1) at (8.5,8)  {$x_{1}$};
  \node[draw,fill=blue!20] (m2) at (13.5,8)  {$x_{2}$};
  \node[draw,fill=blue!20] (m3) at (13.5,3.5) {$x_{3}$};
   \node[draw,fill=blue!20] (m4) at (11,0) {$x_{4}$};
   \node[draw,fill=blue!20] (m5) at (8.5,3.5) {$x_{5}$};
 \node (m6) at (7,2.5){$w(x_{5})\neq 1$};
 \node (m7) at (11,-1.5){$w(x_{4})\neq 1$};
 \node (m8) at (8.5,9){$1$};
 \node (m9) at (13.5,9){$1$};
 \node (m10) at (14.5,3.5){$1$};
 \node[scale=2] (m11) at (11,5.5){$D_{2}$};

  \foreach \from/\to in {m4/m3,m5/m4,m1/m5}
    \draw[->] (\from) -- (\to);
    \draw (m1) -- (m2);
     \draw (m2) -- (m3);
     \node[draw,fill=blue!20] (o1) at (17.5,8)  {$x_{1}$};
  \node[draw,fill=blue!20] (o2) at (22.5,8)  {$x_{2}$};
  \node[draw,fill=blue!20] (o3) at (22.5,3.5) {$x_{3}$};
   \node[draw,fill=blue!20] (o4) at (20,0) {$x_{4}$};
   \node[draw,fill=blue!20] (o5) at (17.5,3.5) {$x_{5}$};
 \node (o6) at (16,2.5){$w(x_{5})\neq 1$};
 \node (o7) at (20,-1.5){$w(x_{4})\neq 1$};
  \node (o8) at (24,2.5){$w(x_{3})\neq 1$};
  \node (o9) at (17.5,9){$1$};
 \node (o10) at (22.5,9){$1$};
\node[scale=2] (o11) at (20,5.5){$D_{3}$};
  \foreach \from/\to in {o4/o3,o5/o4,o1/o5}
    \draw[->] (\from) -- (\to);
    \draw (o1) -- (o2);
     \draw (o2) -- (o3);
\end{tikzpicture}
\end{center}
\begin{center}
\begin{tikzpicture}
  [scale=.4,auto=left,every node/.style={circle,scale=0.5}]
 
  \node[draw,fill=blue!20] (n1) at (-2.5,8)  {$x_{1}$};
  \node[draw,fill=blue!20] (n2) at (2.5,8)  {$x_{2}$};
  \node[draw,fill=blue!20] (n3) at (2.5,3.5) {$x_{3}$};
   \node[draw,fill=blue!20] (n4) at (0,0) {$x_{4}$};
   \node[draw,fill=blue!20] (n5) at (-2.5,3.5) {$x_{5}$};
 \node (n6) at (-4,2.5){$w(x_{5})\neq 1$};
 \node (n7) at (4,2.5){$w(x_{3})\neq 1$};
 \node (n8) at (-2.5,9){$1$};
 \node (n9) at (2.5,9){$w(x_{2})\neq 1$};
 \node (n10) at (0,-1){$1$};
\node[scale=2] (n11) at (0,5.5){$D_{4}$};
  \foreach \from/\to in {n4/n3,n4/n5,n2/n3,n1/n2}
    \draw[->] (\from) -- (\to);
    \draw (n1) -- (n5);
    \end{tikzpicture}
\end{center}
\medskip

\begin{theorem}{\rm (\cite{prt}, Theorem 49)}\label{cyc4}
If $G \simeq C_{n}$, then $I(D_{G})$ is unmixed if and only if one of the following conditions hold:
\begin{enumerate}[(1)]
\item $n = 3$ and there is $x\in V(D_{G})$ such that $w(x)=1$.
\item  $n\in \{4, 7\}$ and the vertices of $V(D_{G})^{+}$ are sinks.
\item  $n = 5$, there is $(x, y)\in E(D_{G})$ with $w(x) = w(y)=1$ and $D_{G}\not\in \{D_{1}, D_{2}, D_{3}\}$.
\item $D_{G}\simeq D_{4}$.
\end{enumerate}
\end{theorem}
\medskip

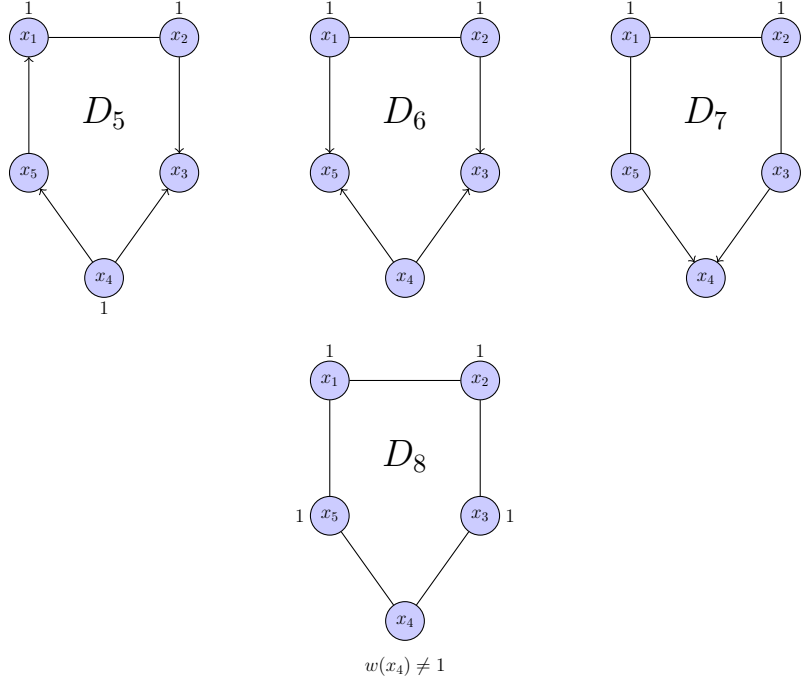
\begin{center}
\begin{tikzpicture}
  [scale=.4,auto=left,every node/.style={circle,scale=0.5}]
 
  \node[draw,fill=blue!20] (n1) at (-2.5,8)  {$x_{1}$};
  \node[draw,fill=blue!20] (n2) at (2.5,8)  {$x_{2}$};
  \node[draw,fill=blue!20] (n3) at (2.5,3.5) {$x_{3}$};
   \node[draw,fill=blue!20] (n4) at (0,0) {$x_{4}$};
   \node[draw,fill=blue!20] (n5) at (-2.5,3.5) {$x_{5}$};
 
 \node (n8) at (-2.5,9){$1$};
 \node (n9) at (2.5,9){$1$};
 \node (n10) at (0,-1){$1$};
\node[scale=2] (n11) at (0,5.5){$D_{5}$};
  \foreach \from/\to in {n4/n3,n4/n5,n2/n3,n5/n1}
    \draw[->] (\from) -- (\to);
    \draw (n1) -- (n2);
    
     \node[draw,fill=blue!20] (m1) at (7.5,8)  {$x_{1}$};
  \node[draw,fill=blue!20] (m2) at (12.5,8)  {$x_{2}$};
  \node[draw,fill=blue!20] (m3) at (12.5,3.5) {$x_{3}$};
   \node[draw,fill=blue!20] (m4) at (10,0) {$x_{4}$};
   \node[draw,fill=blue!20] (m5) at (7.5,3.5) {$x_{5}$};
 
 \node (m8) at (7.5,9){$1$};
 \node (m9) at (12.5,9){$1$};

 \node[scale=2] (m11) at (10,5.5){$D_{6}$};

  \foreach \from/\to in {m4/m3,m4/m5,m1/m5,m2/m3}
    \draw[->] (\from) -- (\to);
    \draw (m1) -- (m2);
    
     \node[draw,fill=blue!20] (o1) at (17.5,8)  {$x_{1}$};
  \node[draw,fill=blue!20] (o2) at (22.5,8)  {$x_{2}$};
  \node[draw,fill=blue!20] (o3) at (22.5,3.5) {$x_{3}$};
   \node[draw,fill=blue!20] (o4) at (20,0) {$x_{4}$};
   \node[draw,fill=blue!20] (o5) at (17.5,3.5) {$x_{5}$};
 
  \node (o9) at (17.5,9){$1$};
 \node (o10) at (22.5,9){$1$};
\node[scale=2] (o11) at (20,5.5){$D_{7}$};
  \foreach \from/\to in {o3/o4,o5/o4}
    \draw[->] (\from) -- (\to);
    \draw (o1) -- (o2);
     \draw (o2) -- (o3);
     \draw (o1) -- (o5);
\end{tikzpicture}
\end{center}
\begin{center}
\begin{tikzpicture}
  [scale=.4,auto=left,every node/.style={circle,scale=0.5}]
 
  \node[draw,fill=blue!20] (n1) at (-2.5,8)  {$x_{1}$};
  \node[draw,fill=blue!20] (n2) at (2.5,8)  {$x_{2}$};
  \node[draw,fill=blue!20] (n3) at (2.5,3.5) {$x_{3}$};
   \node[draw,fill=blue!20] (n4) at (0,0) {$x_{4}$};
   \node[draw,fill=blue!20] (n5) at (-2.5,3.5) {$x_{5}$};
 \node (n6) at (-3.5,3.5){$1$};
 \node (n7) at (3.5,3.5){$1$};
 \node (n8) at (-2.5,9){$1$};
 \node (n9) at (2.5,9){$1$};
 \node (n10) at (0,-1.5){$w(x_{4})\neq 1$};
\node[scale=2] (n11) at (0,5.5){$D_{8}$};
  \foreach \from/\to in {n4/n3,n4/n5,n3/n2,n5/n1}
    \draw[] (\from) -- (\to);
    \draw (n1) -- (n2);
    \end{tikzpicture}
\end{center}
\medskip

\begin{theorem}\label{cyc5}
If $H \simeq C_{5}$, then $I(D_{H})$ is Cohen-Macaulay if and only if $D_{H}\simeq D_{4}$ or there is $(x, y)\in E(D_{H})$ with $w(x) = w(y)=1$ and $D_{H}\not\in \{D_{1}, D_{2}, D_{3}\}$.
\end{theorem}
\begin{proof}
$"\Longrightarrow"$ $I(D_{H})$ is Cohen-Macaulay implies $I(D_{H})$ is unmixed. Therefore by Theorem \ref{cyc4} the "only if part" follows.
\medskip

$"\Longleftarrow"$ To prove the "if part" we will consider some cases and we will not consider the trivial case $D_{H}\simeq C_{5}$.
\medskip

\textbf{Case-1:} $D_{H}\simeq D_{4}$. Now compare $D_{H}$ with the second construction of previous section taking $z=x_{3}$. Then $D_{G}=D_{H}\setminus \{x_{3}\}$ is a weighted oriented path on four vertices and so $D_{G}$ is Cohen-Macaulay by Theorem \ref{cyc1}. Again $D_{F}=D_{G}\setminus \{x_{2},x_{4}\}$ being a weighted oriented path on two vertices is Cohen-Macaulay by Theorem \ref{cyc1}. Note that $\{x_{2},x_{4}\}$ can not be contained in any minimal vertex cover of $D_{G}$ and $$\mathrm{height}\,(\big<I(D_{G}), x_{2},x_{4}\big>)=3=\mathrm{height}\,(I(D_{G}))+1.$$ Hence by Proposition \ref{wocons8} $D_{H}$ is Cohen-Macaulay.
\medskip

If there is $(x, y)\in E(D_{H})$ with $w(x) = w(y)=1$ and $D_{H}\not\in \{D_{1}, D_{2}, D_{3}\}$, then $D_{H}$ is one of $\{D_{5},\,D_{6},\,D_{7},\,D_{8}\}$.
\medskip

\textbf{Case-2:} $D_{H}\simeq D_{5}$. Take $x_{3}=z$ and proceed similar like case-1. Then by Proposition \ref{wocons8} we get $D_{H}$ is Cohen-Macaulay.
\medskip

\textbf{Case-3:}  $D_{H}\simeq D_{6}$. Similarly in this case taking $x_{3}=z$ and using Proposition \ref{wocons8} we get $D_{H}$ is Cohen-Macaulay.
\medskip

\textbf{Case-4:}  $D_{H}\simeq D_{7}$. In this case consider $x_{4}=z$ and use Proposition \ref{wocons8} to we get $D_{H}$ is Cohen-Macaulay.
\medskip

\textbf{Case-5:}  $D_{H}\simeq D_{8}$. We consider some sub-cases and based on that we will choose $z$. For $(x_{4},x_{3})\in E(D_{H})$ take $x_{3}=z$, for $(x_{4},x_{5})\in E(D_{H})$ take $x_{5}=z$, for $(x_{3},x_{4}),(x_{5},x_{4})\in E(D_{H})$ take $x_{4}=z$ and use Proposition \ref{wocons8} to we get $D_{H}$ is Cohen-Macaulay.
\end{proof}
\medskip

Now we know the Conjecture \ref{woconj} mentioned in \cite{prt} is true for weighted oriented bipartite graph from (\cite{hlmrv}, Corollary 5.3), holds for some graphs with perfect matching by Theorem \ref{wocm} and (\cite{gmsvv}, Theorem 5) assures it is true for weighted oriented forest also. By Theorem \ref{cyc3}, Theorem \ref{cyc4}, Theorem \ref{cyc5} we can see that the Conjecture is true for weighted oriented cycles also.

\section{Cohen-Macaulay Alexander Dual of Weighted Oriented Edge Ideals}
\medskip

We are familiar to the concept of Alexander dual of square-free monomial ideals (see \cite{vil}, Definition 6.3.38). The notion of Alexander duality for arbitrary monomial ideals was introduced in \cite{mil}. In this section we will study Alexander dual of edge ideals of weighted oriented graphs. First we shall recall some definitions, notations and results from \cite{mil}:
\medskip

Let $A=K[x_{1},\ldots,x_{n}]$ be the polynomial ring over the field $K$. Now any monomial of $A$ can be determined by an unique vector $\mathbf{a}=(a_{1},\ldots,a_{n})\in \mathbb{N}^{n}$ and an irreducible monomial ideal is uniquely determined by a vector $\mathbf{b}=(b_{1},\ldots,b_{n})\in \mathbb{N}^{n}$. We will use the following notation to represent them:
$$\mathbf{x^{a}}=x_{1}^{a_{1}}\cdots x_{n}^{a_{n}}\hspace{0.5cm}\text{and}\hspace{0.5cm} \mathbf{m^{b}}=\big< x_{i}^{b_{i}}\mid b_{i}\geq 1\big>.$$
For a monomial ideal $I$ we will denote the exponent of the least common multiple of the minimal generators by the vector $\mathbf{a}_{I}$. For two vectors $\mathbf{a,b}\in \mathbb{N}^{n}$ we write $\mathbf{b}\preceq \mathbf{a}$ if each $b_{i}\leq a_{i}$. If $\mathbf{0}\preceq\mathbf{b}\preceq \mathbf{a}$, a vector $\mathbf{b^{a}}$ is defined as it's $i^{\text{th}}$ coordinate is $a_{i}+1-b_{i}$ for $b_{i}\geq 1$ and $0$ else. By (\cite{mil}, Corollary 1.3) we have $\mathbf{b}\preceq \mathbf{a}_{I}$ for any $\mathbf{m^{b}}\in \mathrm{Irr}(I)$, where $\mathrm{Irr}(I)$ denote the set of irredundant irreducible components of $I$ and so the following definition make sense.
\medskip

\begin{definition}[\cite{mil}]\label{alx}
Given a monomial ideal $I$ and $\mathbf{a}\succeq \mathbf{a}_{I}$, the \textbf{Alexander dual}
ideal $I^{\mathbf{a}}$ with respect to $\mathbf{a}$ is defined by
$$I^{\mathbf{a}}=\big< \mathbf{x}^{\mathbf{b}^{\mathbf{a}}}\mid \mathbf{m}^{\mathbf{b}}\in \mathrm{Irr} (I)\big>.$$
For the special case when $\mathbf{a}=\mathbf{a}_{I}$, we write $I^{\vee}=I^{\mathbf{a}_{I}}$.
\end{definition}
\medskip

\begin{proposition}[\cite{mil}, Corollary 2.14]\label{alx1}
$(I^{\mathbf{a}})^{\mathbf{a}}=I$, $(\mathbf{b}^{\mathbf{a}})^{\mathbf{a}}=\mathbf{b}$ and 
$$ I^{\mathbf{a}}=\bigcap\{\mathbf{m}^{\mathbf{b}^{\mathbf{a}}}\mid \mathbf{x}^{\mathbf{b}}\,\, \text{is a minimal generator of}\,\, I\}.$$
\end{proposition}
\medskip

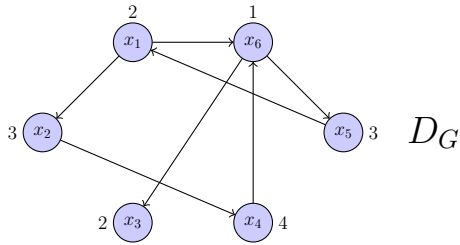
\begin{example}
\end{example}
\begin{center}
\begin{tikzpicture}
  [scale=.4,auto=left,every node/.style={circle,scale=0.5}]
 
  \node[draw,fill=blue!20] (n1) at (3,3)  {$x_{1}$};
  \node[draw,fill=blue!20] (n2) at (0,0)  {$x_{2}$};
  \node[draw,fill=blue!20] (n3) at (3,-3) {$x_{3}$};
   \node[draw,fill=blue!20] (n4) at (7,-3) {$x_{4}$};
   \node[draw,fill=blue!20] (n5) at (10,0) {$x_{5}$};
 \node[draw,fill=blue!20] (n6) at (7,3) {$x_{6}$};
 \node (m1) at (3,4){$2$};
 \node (m2) at (-1,0){$3$};
 \node (m3) at (2,-3){$2$};
 \node (m4) at (8,-3){$4$};
 \node (m5) at (11,0){$3$};
 \node (m6) at (7,4){$1$};
\node[scale=2] (n11) at (13,0){$D_{G}$};
 
  \foreach \from/\to in {n1/n2,n1/n6,n5/n1,n2/n4, n6/n3, n4/n6, n6/n5}
    \draw[->] (\from) -- (\to);
   
\end{tikzpicture}
\end{center}
For the above weighted oriented graph $D_{G}$ we have
{\normalsize
\begin{align*} 
I=I(D_{G})&=\big< x_{1}x_{2}^{3}, x_{5}x_{1}^{2}, x_{1}x_{6}, x_{2}x_{4}^{4}, x_{6}x_{3}^{2}, x_{4}x_{6}, x_{6}x_{5}^{3}\big>\\
&=\big<x_{2},x_{5},x_{6}\big>\cap \big<x_{1},x_{4}^{4},x_{6}\big>\cap \big<x_{1}^{2},x_{2},x_{6}\big>\cap\\
&\hspace{0.6cm} \big<x_{1},x_{3}^{2},x_{4},x_{5}^{3}\big>\cap \big<x_{2}^{3},x_{4}^{4},x_{5},x_{6}\big>\cap \big<x_{1}^{2},x_{2}^{3},x_{4}^{4},x_{6}\big>.
\end{align*}
}
Here $\mathbf{a}_{I}=(2,3,2,4,3,1)$ and so we get
{\normalsize
\begin{align*}
I^{\vee}&=\big< x_{2}^{3}x_{5}^{3}x_{6}, x_{1}^{2}x_{4}x_{6}, x_{1}x_{2}^{3}x_{6}, x_{1}^{2}x_{3}x_{4}^{4}x_{5}, x_{2}x_{4}x_{5}^{3}x_{6}, x_{1}x_{2}x_{4}x_{6}\big>\\
&=\big<x_{1}^{2},x_{2}\big>\cap \big<x_{5}^{3},x_{1}\big>\cap \big<x_{1}^{2},x_{6}\big>\cap \big<x_{2}^{3},x_{4}\big>\cap \\
&\hspace{0.6cm} \big<x_{6},x_{3}\big>\cap \big<x_{4}^{4},x_{6}\big>\cap\big<x_{6},x_{5}\big>.
\end{align*}
}
\medskip

\begin{theorem}[\cite{lg} and \cite{vil}]\label{alx2}
Let $G$ be a graph and let $\overline{G}$ be its complement. Then
$I(G)^{\vee}$ is Cohen-Macaulay if and only if $\overline{G}$ is a chordal graph.
\end{theorem}
\medskip

Note that generalized Alexander dual for any monomial ideal in definition \ref{alx} matches with the Alexander dual of square-free monomial ideal defined in \cite{vil}.
\medskip

\begin{lemma}\label{alx3}
Let $D_{G}$ be a weighted oriented graph on the vertex set $V(D_{G})=\{x_{1}\ldots,x_{n}\}$. We construct a simple graph $G^{D}$ such that $V(G^{D})=\{x_{1,1}, \ldots, x_{n,1}\}\cup \{x_{k,2},\ldots,x_{k,w(k)}\mid x_{k}\,\text{ is not sink and}\,\, w(k)\neq 1\}$ and $$E(G^{D})=\bigcup_{(x_{i},x_{j})\in E(D_{G})}\{ \{x_{i,1},x_{j,1}\},\ldots, \{x_{i,w(i)},x_{j,1}\}\}.$$ Then $(I^{\vee}(\mathrm{pol}))^{\vee}=I(G^{D})$, where $I=I(D_{G})$.
\end{lemma}

\begin{proof}
Let $\mathcal{C}_{s}$ be the set of strong vertex covers of $D_{G}$. Then by  Theorem \ref{prmdc} we have $I=I(D_{G})=\bigcap_{C\in\mathcal{C}_{s}} I_{C}$
is a primary decomposition of $J$, where $$I_{C}=\big < L_{1}(C)\cup \{x_{j}^{w(j)}\mid x_{j}\in L_{2}(C)\cup L_{3}(C)\}\big>.$$
Now let $I_{C}=\mathbf{m}^{\mathbf{b}_{C}}$ and take $\mathbf{a}=\mathbf{a}_{I}=(w(1),\ldots,w(n))$. Then we can write 
$$I^{\vee}=I^{\mathbf{a}}=\big < \mathbf{x}^{\mathbf{b}^{\mathbf{a}}_{C}}\mid \mathbf{m}^{\mathbf{b}_{C}}=I_{C}\,\,\text{for some}\, C\in\mathcal{C}_{s}\big>. $$
From Proposition \ref{alx1} we have 
$$ I^{\vee}=\bigcap\{ \mathbf{m}^{\mathbf{b}^{\mathbf{a}}}\mid \mathbf{x^{b}}\,\,\text{is a minimal generator of}\,\, I\}.$$
Observe that $\mathbf{x^{b}}=x_{i}x_{j}^{w(j)}$ implies $\mathbf{m^{b^{a}}}=\big< x_{i}^{w(i)}, x_{j}\big>$ and so $I^{\vee}=\bigcap_{(x_{i},x_{j})\in E(D_{G})}\big< x_{i}^{w(i)}, x_{j}\big>$. Now, by (\cite{far}, Proposition 2.5) we have
$$I^{\vee}(\mathrm{pol})=\bigcap_{(x_{i},x_{j})\in E(D_{G})}\big(\big< x_{i,1},x_{j,1}\big>\cap\cdots\cap\big< x_{i,w(i)},x_{j,1}\big>\big)$$
and hence 
\begin{align*}
I^{\vee}(\mathrm{pol})^{\vee}&=\big<\{ x_{i,1}x_{j,1},\ldots, x_{i,w(i)}x_{j,1}\mid (x_{i},x_{j})\in E(D_{G})\}\big>\\ &= I(G^{D})
\end{align*}
\end{proof}
\medskip

Let $D_{G}$ be a weighted oriented graph such that $\overline{G}$ is chordal. Without loss of generality we can suppose $(x_{1},\ldots,x_{n})$ is a perfect elimination ordering for $\overline{G}$. We say $D_{G}$ satisfies property $(\ast)$ if the following conditions $(\ast)$ holds:

\begin{equation}
\text{\parbox{.9\textwidth}
 {If $x_{j}$ appears before $x_{i}$ in the ordering and $\{x_{j},x_{i}\}\not\in E(G)$, then for any non-sink $x_{k}$ with $(x_{k},x_{i})\not\in E(D_{G})$ and $w(k)\neq 1$ we should have $(x_{k},x_{j})\not\in E(D_{G})$.}}\tag{$\ast$}
 \end{equation}
 \medskip
 
 Our next two results Theorem \ref{alx4} and Theorem \ref{alx5} give some conditions for Cohen-Macaulay property of the Alexander dual of the ideal $I(D_{G})$, which can be seen as a generalization of the Theorem \ref{alx2}.
 \medskip
 
 \begin{theorem}\label{alx4}
 Let $D_{G}$ be a weighted oriented graph. If $\overline{G}$ is chordal and $D_{G}$ satisfies the property $(\ast)$ , then $I^{\vee}$ is Cohen-Macaulay, where $I=I(D_{G})$ . 
 \end{theorem}
 
 \begin{proof}
 Given $\overline{G}$ is chordal and so $\overline{G}$ satisfies a perfect elimination ordering, say $(x_{1},\ldots,x_{n})$. Let $S=\{x_{i_{1}},x_{i_{2}},\ldots, x_{i_{s}}\}$ be the set of non-sink vertices with $w(i_{j})\neq 1$ for $1\leq j\leq s$ and $T= \{x_{i_{1},2},\ldots,x_{i_{1},w(i_{1})},\ldots, x_{i_{s},2},\ldots,x_{i_{s},w(i_{s})}\}$.
 \medskip
 
\textbf{Claim:} {\fontsize{13.2}{14}\selectfont $(x_{i_{1},2},\ldots,x_{i_{1},w(i_{1})},\ldots, x_{i_{s},2},\ldots,x_{i_{s},w(i_{s})}, x_{1,1}, \ldots, x_{n,1} )$} is a perfect elimination ordering for $\overline{G^{D}}$.
\medskip

Since $\{x_{i_{p},l},x_{i_{q},k}\}\not \in E(G^{D})$ for $1\leq p,q\leq s$ and $2\leq l\leq w(i_{p}),\, 2\leq k\leq w(i_{q})$, the subgraph induced by $(\{x_{i_{r},j}\}\cup \mathcal{N}(x_{i_{r},j}))\cap T$ in $\overline{G^{D}}$ is complete. Therefore every $x_{i_{r},j}\in T$ is simplicial in the subgraph induced by the vertices $$ \{x_{i_{1},2},\ldots,x_{i_{1},w(i_{1})},\ldots, x_{i_{r},2},\ldots,x_{i_{r},j}\}.$$
Again $(x_{1},\ldots,x_{n})$ is a perfect elimination ordering for $\overline{G}$ implies $(\{x_{j,1}\}\cup\mathcal{N}(x_{j,1}))\cap \{ x_{1,1},\ldots,x_{j,1}\}$ induced a complete subgraph of $\overline{G^{D}}$ for each $1\leq j\leq n$. Let $$\mathcal{N}(x_{j,1})\cap \{ x_{1,1},\ldots,x_{j,1}\}=J$$ in $\overline{G^{D}}$. Suppose $\{x_{i_{p},l},x_{j,1}\}\in E(\overline{G^{D}})$ for some $x_{i_{p},l}\in T$. So we have $\{x_{i_{p},l},x_{j,1}\}\not\in E(G^{D})$ i.e., $(x_{i_{p}},x_{j})\not\in E(D_{G})$. By property $(\ast)$, we have for every $x_{m,1}\in J$
\begin{align*}
&(x_{i_{p}},x_{m,1})\not\in E(D_{G})\\
\Rightarrow\hspace{0.5cm} & \{x_{i_{p},l},x_{m,1}\}\not\in E(G^{D}),\,\, \text{for}\,\, 2\leq l\leq w(i_{p})\\
\Rightarrow\hspace{0.5cm} & \{x_{i_{p},l},x_{m,1}\}\in E(\overline{G^{D}})
\end{align*}
Therefore $x_{j,1}$ is simplicial in the subgraph induced by $T\cup\{x_{1,1},\ldots,x_{j,1}\}$ in $\overline{G^{D}}$. This is true for all $1\leq j\leq n$ and so the claim is proved. Hence $\overline{G^{D}}$ is chordal and by Theorem \ref{alx2} $I(G^{D})^{\vee}$ is Cohen-Macaulay. Now Lemma \ref{alx3} gives $I^{\vee}(\mathrm{pol})$ is Cohen-Macaulay and hence from Theorem \ref{cmpol} we have $I^{\vee}$ is Cohen-Macaulay. 
 \end{proof}
\medskip

\begin{theorem}\label{alx5}
Let $D_{G}$ be a weighted oriented graph. If $I(D_{G})^{\vee}$ is Cohen-Macaulay, then $\overline{G}$ is chordal.
\end{theorem}

\begin{proof}
Since $I(D_{G})^{\vee}$ is Cohen-Macaulay, by Theorem \ref{cmpol} we have  $I(D_{G})^{\vee}(\mathrm{pol})$ is Cohen-Macaulay. By Lemma \ref{alx3} $I(G^{D})^{\vee}$ is Cohen-Macaulay and hence from Theorem \ref{alx2} we have $\overline{G^{D}}$ is chordal. Therefore $\overline{G}$ is chordal as $G$ is an induced subgraph of $G^{D}$.
\end{proof}
\medskip

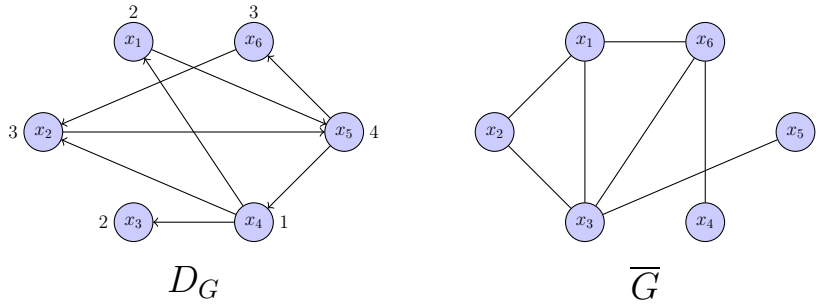
\begin{example}
\end{example}
\begin{center}
\begin{tikzpicture}
  [scale=.4,auto=left,every node/.style={circle,scale=0.5}]
 
  \node[draw,fill=blue!20] (n1) at (3,3)  {$x_{1}$};
  \node[draw,fill=blue!20] (n2) at (0,0)  {$x_{2}$};
  \node[draw,fill=blue!20] (n3) at (3,-3) {$x_{3}$};
   \node[draw,fill=blue!20] (n4) at (7,-3) {$x_{4}$};
   \node[draw,fill=blue!20] (n5) at (10,0) {$x_{5}$};
 \node[draw,fill=blue!20] (n6) at (7,3) {$x_{6}$};
 \node (m1) at (3,4){$2$};
 \node (m2) at (-1,0){$3$};
 \node (m3) at (2,-3){$2$};
 \node (m4) at (8,-3){$1$};
 \node (m5) at (11,0){$4$};
 \node (m6) at (7,4){$3$};
\node[scale=2] (n7) at (5,-5){$D_{G}$};
 
  \foreach \from/\to in {n1/n5,n4/n1,n6/n2,n2/n5, n4/n2, n4/n3, n5/n4, n5/n6}
    \draw[->] (\from) -- (\to);
    
     \node[draw,fill=blue!20] (n1) at (18,3)  {$x_{1}$};
  \node[draw,fill=blue!20] (n2) at (15,0)  {$x_{2}$};
  \node[draw,fill=blue!20] (n3) at (18,-3) {$x_{3}$};
   \node[draw,fill=blue!20] (n4) at (22,-3) {$x_{4}$};
   \node[draw,fill=blue!20] (n5) at (25,0) {$x_{5}$};
 \node[draw,fill=blue!20] (n6) at (22,3) {$x_{6}$};
 \node[scale=2] (n7) at (20,-5){$\overline{G}$};
   
    \foreach \from/\to in {n1/n2,n3/n1,n3/n2,n1/n6, n3/n6, n4/n6, n5/n3}
    \draw[] (\from) -- (\to);
    
\end{tikzpicture}
\end{center}
In the above graph $D_{G}$, it is clear that $\overline{G}$ is chordal and $D_{G}$ satisfies the property $(\ast)$ with respect to the perfect elimination ordering $(x_{1},x_{3},x_{6},x_{2},x_{4},x_{5})$. Therefore by Theorem \ref{alx4} $I(D_{G})^{\vee}$ is Cohen-Macaulay.
\medskip

\bibliographystyle{amsalpha}

\end{document}